\documentclass[10pt,a4paper]{article}
\usepackage{amsthm,amssymb,mathrsfs,setspace,amsmath}
\DeclareMathOperator{\sgn}{sgn}
\usepackage[a4paper,bindingoffset=0.2in,
left=0.5in,right=0.5in,
top=1in,bottom=1in,footskip=.25in]{geometry}
\title{Certain properties of continuous fractional wavelet transform on Hardy space and Morrey space}
\author{Amit K. Verma$^a$\thanks{Corresponding author email: akverma@iitp.ac.in}, Bivek Gupta$^b$\\{\small\textit{$^{a,b}$ Department of Mathematics, IIT Patna, Bihta, Patna 801103.}}}
\theoremstyle{definition}
\newtheorem{defn}{Definition}[section]
\newtheorem{lemma}{Lemma}[section]
\newtheorem{remark}{Remark}[section]
\newtheorem{theorem}{Theorem}[section]
\newtheorem{corollary}{Corollary}[section]
\begin{document}
\maketitle

\begin{abstract}
In this paper we define a new class of continuous fractional wavelet transform (CFrWT) and study its properties in Hardy space and Morrey space. The theory developed generalize and complement some of already existing results.
\end{abstract}
\section{Introduction}
Even though the classical wavelet transform (WT) serves as a powerful tool in signal processing and analysis, its analyzing capability is limited to the time-frequency plane. Fractional Fourier transform (FrFT)(\cite{almeida1994fractional},\cite{mcbride1987namias},\cite{namias1980fractional}) gives the fractional Fourier domain (FrFD) frequency content of the signal, but it fails in giving the local information of the signal. Mendlovic et al. (\cite{mendlovic1997fractional}), first introduced the FrWT to deal with the optical signals. They first derive the fractional spectrum of the signal by using the FrFT and performed the WT of the fractional spectrum. But the transform defined in such a way, fails in giving the information about the local property of the signal, since the FrFT gives the fractional frequency of the signal during the entire duration of the signal rather than for a particular time, and the fractional spectrum of the signal cannot be ascertained when those fractional frequencies exist.

The novel Fractional wavelet transform (FrWT) based on fractional convolution was proposed by Shi et al. (\cite{shi2012novel}). They studied basic properties of the FrWT like inner product theorem, Parseval's relation and inversion formula for the function in $L^2(\mathbb{R}).$ Prasad et al. (\cite{prasad2014generalized}) studied some properties of FrFT such as Riemann-Lebesgue lemma. Also, they extended  the inner product theorem of the CFrWT, studied in \cite{shi2012novel}, in the context of two fractional wavelets. Dai et al. (\cite{dai2017new}) proposed a new type of FrWT and obtained the associated multiresolution analysis (MRA). This is more general than the transforms defined in \cite{prasad2014generalized} and \cite{shi2012novel}. It displays the time and FrFD-frequency information jointly in the time-FrFD-frequency plane.

Luchko et al. (\cite{luchko2008fractional}) introduced a new FrFT and implemented this theory on the Lizorkin space, and also discussed many important results  involving fractional derivatives. To know more about the FrFT reader may follow \cite{kilbas2010fractional},\cite{srivastava2017family}. In (\cite{srivastava2019certain,upadhyayCFrWT}), authors studied the new theory of FrWT, associated with the FrFT given in \cite{kilbas2010fractional, srivastava2017family, luchko2008fractional}, and obtained some of its properties like inner product relation, inversion formula, etc. They also discussed MRA associated with this FrWT, along with the construction of the orthogonal fractional wavelets. This theory can also be used in the study of quantum mechanics, signal processing and other areas of science and engineering.

Several important function spaces like Besov, Sobolev, Holder, Zygmund, BMO, etc are given characterization in terms of wavelets involved in the classical WT (\cite{daubechies1992ten},\cite{meyer1992wavelets}). WT has also been studied in various function spaces and the spaces of distributions (\cite{rieder1990wavelet},\cite{pathak2004wavelet},\cite{pathak2007boundedness}). Chuong et al. \cite{chuong2013boundedness} studied the boundedness of the WT on the Besov, BMO and Hardy spaces. Furthermore, for the compactly supported basic wavelet, the boundedness of the WT is also established on the weighted Besov space and weighted BMO space associated with the tempered weight function. In the recent years, Prasad and Kumar (\cite{prasad2015continuous},\cite{prasad2016continuous},\cite{prasad2016fractional}) discussed the CFrWT on the  generalized weighted Sobolev spaces and some function spaces and obtained its boundedness. Not only that, the WT and CFrWT have also been studied by many authors on some spaces of test functions, like Gelfand-Shilov spaces (\cite{pathak2006wavelet},\cite{pilipovic2016wavelet},\cite{prasad2012fractional},\cite{prasad2013fractional}). Based on the convolution of linear canonical transform (LCT)(\cite{wei2014generalized}), Guo et al. (\cite{guo2018linear}) proposed a linear canonical wavelet transform (LCWT), which is a generalization of the transform studied in \cite{prasad2014generalized}. The authors also proved the continuity of this transform on some space of test functions and the generalized Sobolev space. To know more about the literature, reader can read the references and the references therein.

Motivated by above works we have studied the CFrWT. We complement the theory of CFrWT studied in (\cite{srivastava2019certain, upadhyayCFrWT}) by adding some new results and  studying its properties in Hardy and Morrey spaces. We present the orthogonality relation which helps us to  conclude that the images of the CFWTs associated with two different fractional wavelets are orthogonal if the respective argument functions are orthogonal. Also, we establish the reconstruction formula and the characterization of the range of CFrWT based on two fractional wavelets.  Moreover, we derive the formulas for the CFrWT associated with  the convolution and correlation of two functions. Furthermore,  we  study the boundedness of the CFrWT on $H^1(\mathbb{R})~ \mbox{and}~L_M^{1,\nu}(\mathbb{R})$  and also study the dependence of the CFrWT on its wavelet and the argument function by determining the $H^1(\mathbb{R})$ and $L_M^{1,\nu}(\mathbb{R})-$distance of two CFrWTs with different fractional wavelets and argument functions.

The organization of the paper is as follows:
In section 2, we recall some basic definitions and results. In section 3, we have derived the orthogonality relation, the reconstruction formula and characterized the range of the transform in the context of two fractional wavelets. Also, we have derived the formulas for the CFrWT when the argument function or fractional wavelet is a convolution or correlation of two functions. Section 4 is further divided into two subsections. In each of these two subsections  the boundedness of CFrWT on Hardy space $H^1(\mathbb{R})$ and Morrey space $L_M^{1,\nu}(\mathbb{R})$ along with its approximation properties are studied. Finally, we end this paper by the conclusions in section 5.
\section{Preliminaries}
In this section we recall some existing definitions and results that we be used in this paper.

\begin{defn}\label{P2defn2.1}
 The convolution of complex-valued measurable functions $f$ and $g$ defined on $\mathbb{R},$ is given by
\begin{equation}
(f\star g)(x)=\int_{\mathbb{R}}f(u)g(x-u)du,~x\in\mathbb{R}
\end{equation}
whenever the integral is well-defined.
\end{defn}
\begin{defn}\label{P2defn2.2}
The correlation of complex-valued measurable functions $f$ and $g$ defined on $\mathbb{R},$ is given by
\begin{equation}(f\circ g)(x)=\int_{\mathbb{R}}\overline{f(u)}g(x+u)du,~x\in\mathbb{R}
\end{equation} 
whenever the integral is well-defined.
\end{defn}

\begin{defn}\label{P2defn2.3}
\cite{srivastava2019certain} The fractional Fourier transform (FrFT), of real order $\theta~(0<\theta\leq 1),$ of a function $f\in L^2(\mathbb{R})$ is defined by 
\begin{equation}\label{P2eqndef}
(\mathfrak{F}_\theta f)(\xi)=\int_{\mathbb{R}}e^{-i(\sgn\xi)|\xi|^{\frac{1}{\theta}}t}f(t)dt,~\xi\in\mathbb{R}.
\end{equation}
For $\theta=1,$ the fractional Fourier transform defined in (\ref{P2eqndef}) reduces to the classical Fourier transform.\\
The corresponding inverse fractional Fourier transform is defined as follows:
$$f(t)=\frac{1}{2\pi\theta}\int_{\mathbb{R}}e^{i(\sgn\xi)|\xi|^{\frac{1}{\theta}}t}(\mathfrak{F}_\theta f)(\xi)|\xi|^{\frac{1}{\theta}-1}d\xi.$$
\end{defn}
\begin{lemma}\label{P2lemma2.1}
Let $\psi\in L^2(\mathbb{R}),$ then
\begin{equation}\label{P2FrFTofDW}
(\mathfrak{F}_\theta\psi_{a,b,\theta})(\xi)=|a|^{\frac{1}{2\theta}}e^{-i (\sgn\xi)|\xi|^{\frac{1}{\theta}}b}(\mathfrak{F}_\theta\psi)(a\xi),
\end{equation}
where
 \begin{equation}\label{P2eqn14.}
\psi_{a,b,\theta}(t)=\frac{1}{|a|^{\frac{1}{2\theta}}}\psi\left(\frac{t-b}{(\sgn a)|a|^{\frac{1}{\theta}}}\right),~a,b\in\mathbb{R}.
\end{equation}
\end{lemma}
\begin{proof}
Refer \cite[Page 7]{srivastava2019certain}.
\end{proof}

\section{CFrWT} 
Before we begin with the definition of the CFrWT we recall the definition of the fractional wavelet given by Srivastava et al. (\cite{srivastava2019certain}). We then prove a theorem that helps in constructing a family of fractional wavelets from a given one.
\begin{defn}\label{defnAC}
A fractional wavelet is a non-zero function $\psi\in L^1(\mathbb{R})\cap L^2(\mathbb{R}),$ satisfying 
\begin{equation}\label{P2eqn15}
C_{\psi,\theta}:=\int_{\mathbb{R}}\frac{|\mathfrak{F}_\theta\psi(\xi)|^2}{|\xi|}d\xi<\infty.
\end{equation} 
\end{defn}
 Now, we prove the following  theorem which indicate the  construction of a family of fractional wavelets from a given one.
\begin{theorem}\label{P2theo2.5}
Let $\psi$ be a fractional wavelet and $\phi$ be a function  in $L^1(\mathbb{R}),$ then $\psi\star\phi$ and $\psi\circ\phi$ are also fractional wavelets.
\end{theorem}
\begin{proof}
Since $\psi\in L^1(\mathbb{R})\cap L^2(\mathbb{R})$ and $\phi\in L^1(\mathbb{R}),$ $\psi\star\phi\in L^1(\mathbb{R})\cap L^2(\mathbb{R}).$\\
Now,
\begin{eqnarray*}
\int_{\mathbb{R}}\frac{|\mathfrak{F}_\theta(\psi\star\phi)(\xi)|^2}{|\xi|}d\xi &=& \int_{\mathbb{R}}\frac{|(\mathfrak{F}_\theta\psi)(\xi)|^2|(\mathfrak{F}_\theta\phi)(\xi)|^2}{|\xi|}d\xi,~\mbox{since}~\mathfrak{F}_\theta(\psi\star\phi)(\xi)=  (\mathfrak{F}_\theta\psi)(\xi)(\mathfrak{F}_\theta\phi)(\xi)\\
&\leq &\|\phi\|^2_{L^1(\mathbb{R})}\int_{\mathbb{R}}\frac{|(\mathfrak{F}_\theta\psi)(\xi)|^2}{|\xi|}d\xi,~\mbox{since}~\|\mathfrak{F}_\theta \phi\|_{L^\infty(\mathbb{R})}\leq\|\phi\|_{L^2(\mathbb{R})}\\
&=& \|\phi\|^2_{L^1(\mathbb{R})}C_{\psi,\theta}.
\end{eqnarray*}
Since $\psi$ is a fractional wavelet and $\phi\in L^1(\mathbb{R}),$ 
$$\int_{\mathbb{R}}\frac{|\mathfrak{F}_\theta(\psi\star\phi)(\xi)|^2}{|\xi|}d\xi<\infty.$$
Hence by definition \ref{defnAC}, $\psi\star\phi$ is a fractional wavelet. Similarly, it can be shown that $\psi\circ\phi$ is also a fractional wavelet. This completes the proof.
\end{proof}

\begin{defn}
\cite{srivastava2019certain} The CFrWT of $f$ with respect to a fractional wavelet $\psi$ is defined by 
\begin{equation}\label{P2eqn16.}
\left(W_\psi^\theta f\right)(b,a)=\int_{\mathbb{R}}f(t)\overline{\psi_{a,b,\theta}(t)}dt,~a,b\in\mathbb{R},
\end{equation}
provided the integral is well-defined. Here $\psi_{a,b,\theta}$ is given by equation (\ref{P2eqn14.}).
\end{defn}
We derive some new results of the CFrWT and also generalized  some existing results in the context of two fractional wavelets. We omit the proof which are similar as in \cite{srivastava2019certain} and \cite{upadhyayCFrWT}.

If $f,~g\in L^2(\mathbb{R})$ are orthogonal then the image $W_\psi^\theta f$ and $W_\psi^\theta g$ are also orthogonal in $L^2\left(\mathbb{R}\times\mathbb{R},\frac{dbda}{|a|^{\frac{1}{\theta}+1}}\right)$. This fact is observed by the orthogonality relation for the CFrWT given in \cite{srivastava2019certain}. But this relation is not enough to conclude the orthogonality of $W_\psi^\theta f$ and $W_\phi^\theta g$  for two different fractional wavelets $\psi$ and $\phi$. So in this regard we introduce a more general version of orthogonality relation.  We also  derive reconstruction formula and characterized its range. For  the case $\psi=\phi$, our results coincide with the results in \cite{srivastava2019certain}.
\begin{theorem}\label{P2OR}(Orthogonality relation)
If the fractional wavelets $\phi$ and $\psi$ satisfies
\begin{equation}\label{P2eqn19}
\int_{\mathbb{R}}|(\mathfrak{F}_\theta\phi)(u)|~|(\mathfrak{F}_\theta\psi)(u)|\frac{1}{|u|}du<\infty,
\end{equation}
then for $f,g\in L^2(\mathbb{R}),$
$$\int_{\mathbb{R}}\int_{\mathbb{R}}\left(W_\phi^\theta f\right)(b,a)\overline{\left(W_\psi^\theta g\right)(b,a)}\frac{dbda}{|a|^{\frac{1}{\theta}+1}}=C_{\phi,\psi,\theta}\langle f,g\rangle_{L^2(\mathbb{R})},$$ where 
\begin{equation}\label{P2eqn20}
C_{\phi,\psi,\theta}=\int_{\mathbb{R}}\overline{(\mathfrak{F}_\theta\phi)(u)}(\mathfrak{F}_\theta\psi)(u)\frac{1}{|u|}du.
\end{equation}
\end{theorem}
\begin{proof}
The proof is similar as in \cite{srivastava2019certain, upadhyayCFrWT}.
\end{proof}
\begin{corollary}
Let $\phi,\psi$ be two fractional wavelets and are such that they satisfies the hypothesis of theorem \ref{P2OR}. If further $C_{\phi,\psi,\theta}=0,$ where $C_{\phi,\psi,\theta}$ is given by equation (\ref{P2eqn20}). Then  $W_\phi^\theta \left(L^2(\mathbb{R})\right)$ and $W_\psi^\theta \left(L^2(\mathbb{R})\right)$ are orthogonal.
\end{corollary}
\begin{proof}
Proof follows from results of theorem \ref{P2OR}. 
\end{proof}
\begin{theorem}(Reconstruction formula)
Let $f\in L^2(\mathbb{R})$ and $\phi,\psi$ be two fractional wavelets satisfying (\ref{P2eqn19}) and  $C_{\phi,\psi,\theta},$ as defined in (\ref{P2eqn20}), is non-zero. Then
$$f(t)=\frac{1}{C_{\phi,\psi,\theta}}\int_{\mathbb{R}}\int_{\mathbb{R}}\psi_{a,b,\theta}(t)\left(W_\phi^\theta f\right)(b,a)\frac{dbda}{|a|^{\frac{1}{\theta}+1}}.$$
\end{theorem}
\begin{proof}
The proof is similar as in  \cite{srivastava2019certain, upadhyayCFrWT}.
\end{proof}
\begin{theorem}(Characterization of the range)
Let $C_{\phi,\psi,\theta}$ as defined in (\ref{P2eqn20}), for two fractional wavelets satisfying (\ref{P2eqn19}), is non-zero. Then $F\in L^2\left(\mathbb{R}\times\mathbb{R},\frac{dbda}{|a|^{\frac{1}{\theta}+1}}\right)$ is a CFrWT, with respect to $\phi,$ of some $f\in L^2(\mathbb{R})$ $\it{iff}$
\begin{equation}\label{P2rkhs}
F(b_0,a_0)=\int_{\mathbb{R}}\int_{\mathbb{R}}F(b,a)K_{\phi,\psi,\theta}(b_0,a_0;b,a)\frac{dbda}{|a|^{\frac{1}{\theta}+1}},~(b_0,a_0)\in\mathbb{R}\times\mathbb{R},
\end{equation}
where $K_{\phi,\psi,\theta}$ is the reproducing kernel given by
\begin{equation}\label{P2RKHS}
K_{\phi,\psi,\theta}(b_0,a_0;b,a)=\frac{1}{C_{\phi,\psi,\theta}}\int_{\mathbb{R}}\psi_{a,b,\theta}(t)\overline{\phi_{a_0,b_0,\theta}(t)}dt.
\end{equation}
Moreover, in such a case the kernel is pointwise bounded:
$$|K_{\phi,\psi,\theta}(b_0,a_0;b,a)|\leq\frac{1}{C_{\phi,\psi,\theta}}\|\phi\|_{L^2(\mathbb{R})}\|\psi\|_{L^2(\mathbb{R})}.$$ 
\end{theorem}
\begin{proof}
The proof is similar as in \cite{srivastava2019certain, upadhyayCFrWT}.
\end{proof}
Now, we prove the theorem that gives the formula for the wavelet transform of the convolution and correlation of two functions.
\begin{theorem}
Let  $f\in L^1(\mathbb{R}),~g\in L^2(\mathbb{R})$ and $\psi$ be a fractional wavelet, then
\begin{eqnarray*}
\left(W_\psi^\theta(f\star g)\right)(b,a)&=&\left(f(\cdot)\star(W_\psi^\theta g)(\cdot,a)\right)(b)\\
\mbox{and}~ \left(W_\psi^\theta(f\circ g)\right)(b,a)&=&\left(f(\cdot)\circ(W_\psi^\theta g)(\cdot,a)\right)(b).
\end{eqnarray*}
\begin{proof}
We have
\begin{eqnarray*}
\left(W_\psi^\theta(f\star g)\right)(b,a)&=&\int_{\mathbb{R}}(f\star g)(t)\overline{\psi_{a,b,\theta}(t)}dt\\
&=&\int_{\mathbb{R}}\left\{\int_{\mathbb{R}}f(y)g(t-y)dy\right\}\frac{1}{|a|^\frac{1}{2\theta}}\overline{\psi\left(\frac{t-b}{(\sgn a)|a|^\frac{1}{\theta}}\right)}dt,~\mbox{using definition \ref{P2defn2.1}}\\
&=&\int_{\mathbb{R}}f(y)\left\{\int_{\mathbb{R}}g(t)\overline{\psi_{a,b-y,\theta}}dt\right\}dy\\
&=&\int_{\mathbb{R}}f(y)(W_\psi^\theta g)(b-y,a)dy.
\end{eqnarray*}
Therefore,
$$\left(W_\psi^\theta(f\star g)\right)(b,a)=\left(f(\cdot)\star(W_\psi^\theta g)(\cdot,a)\right)(b).$$
Similarly, it can be shown that
$$\left(W_\psi^\theta(f\circ g)\right)(b,a)=\left(f(\cdot)\circ(W_\psi^\theta g)(\cdot,a)\right)(b).$$
This completes the proof.
\end{proof}
The following theorem gives the expression of the CFrWT when the fractional wavelet associated with the transform is the convolution or the correlation of two functions.
\begin{theorem}
Let $f\in L^1(\mathbb{R})~g\in L^2(\mathbb{R})$ and $\psi$ be a fractional wavelet,  then
\begin{eqnarray*}
\left(W_{f\star\psi}^\theta g\right)(b,a)&=&\frac{1}{|a|^\frac{1}{\theta}}\left(f\left(\frac{\cdot}{(\sgn a)|a|^\frac{1}{\theta}}\right)\circ(W_\psi^\theta g)(\cdot,a)\right)(b)\\
\mbox{and}~ \left(W_{f\circ \psi}^\theta g\right)(b,a)&=&\frac{1}{|a|^\frac{1}{\theta}}\left(f\left(\frac{\cdot}{(\sgn a)|a|^\frac{1}{\theta}}\right)\star(W_\psi^\theta g)(\cdot,a)\right)(b).
\end{eqnarray*}
\end{theorem}
\begin{proof}
Since $f\in L^1(\mathbb{R})$ and $\psi$ is a fractional wavelet, by theorem \ref{P2theo2.5}, $f\star\psi$ is a wavelet.\\
Now,
\begin{eqnarray*}
\left(W_{f\star\psi}^\theta g\right)(b,a)&=&\int_{\mathbb{R}}g(t)\overline{(f\star \psi)_{a,b,\theta}(t)}dt\\
&=&\int_{\mathbb{R}}g(t)\left\{\frac{1}{|a|^\frac{1}{2\theta}}\overline{\int_{\mathbb{R}}f(y)\psi\left(\frac{t-b}{(\sgn a)|a|^\frac{1}{\theta}}-y\right)dy}\right\}dt\\
&=&\int_{\mathbb{R}}\overline{f(y)}\left\{\int_{\mathbb{R}}g(t)\frac{1}{|a|^\frac{1}{2\theta}}\overline{\psi\left(\frac{t-(b+y(\sgn a)|a|^\frac{1}{\theta})}{(\sgn a)|a|^\frac{1}{\theta}}\right)dt}\right\}dy\\
&=&\int_{\mathbb{R}}\overline{f(y)}(W_\psi^\theta g)(b+(\sgn a)|a|^\frac{1}{\theta}y,a)dy.
\end{eqnarray*}
Therefore,
$$\left(W_{f\star\psi}^\theta g\right)(b,a)=\frac{1}{|a|^\frac{1}{\theta}}\left(f\left(\frac{\cdot}{(\sgn a)|a|^\frac{1}{\theta}}\right)\circ(W_\psi^\theta g)(\cdot,a)\right)(b).$$
Again by theorem \ref{P2theo2.5}, $f\circ\psi$ is a wavelet. Proceeding similarly as above it can be shown that
$$\left(W_{f\circ \psi}^\theta g\right)(b,a)=\frac{1}{|a|^\frac{1}{\theta}}\left(f\left(\frac{\cdot}{(\sgn a)|a|^\frac{1}{\theta}}\right)\star(W_\psi^\theta g)(\cdot,a)\right)(b).$$
This completes the proof.
\end{proof}
\end{theorem}
\begin{theorem}
Let $f,g\in L^2(\mathbb{R})$ and $\phi,\psi$ be two fractional wavelets, then 
$$\int_{\mathbb{R}}|b|^{\frac{1}{\theta}-1}(W_\phi^\theta f)(b,a)\overline{(W_\psi^\theta g)(b,a)}db=\frac{|a|^\frac{1}{\theta}}{4\pi^2\theta^2}\langle P_\theta,Q_\theta\rangle_{L^2(\mathbb{R})},$$
where\\
$P_\theta(\xi)=|\xi|^{\frac{1}{\theta}-1}(\mathfrak{F}_\theta f)(\xi)\overline{(\mathfrak{F}_\theta\phi)(a\xi)}~\mbox{and}~Q_\theta(\xi)=|\xi|^{\frac{1}{\theta}-1}(\mathfrak{F}_\theta g)(\xi)\overline{(\mathfrak{F}_\theta\psi)(a\xi)}.$
\end{theorem}
\begin{proof}
\begin{eqnarray*}
\int_{\mathbb{R}}|b|^{\frac{1}{\theta}-1}(W_\phi^\theta f)(b,a)\overline{(W_\psi^\theta g)(b,a)}db&=&\int_{\mathbb{R}}|b|^{\frac{1}{\theta}-1}\langle f,\phi_{a,b,\theta}\rangle_{L^2(\mathbb{R})}\overline{\langle g,\psi_{a,b,\theta}\rangle_{L^2(\mathbb{R})}}db
\end{eqnarray*}
Using Parseval's formula \cite[Theorem 1]{srivastava2019certain}, we have
\begin{align}\label{P2eqn22}
&\hspace{-1.2cm}\displaystyle\int_{\mathbb{R}}|b|^{\frac{1}{\theta}-1}(W_\phi^\theta f)(b,a)\overline{(W_\psi^\theta g)(b,a)}db\notag\\
&=\displaystyle\left(\frac{1}{2\pi\theta}\right)^2\int_{\mathbb{R}}|b|^{\frac{1}{\theta}-1}\left\langle |\cdot|^{\frac{1}{\theta}-1}(\mathfrak{F}_\theta f)(\cdot),(\mathfrak{F}_\theta\phi_{a,b,\theta})(\cdot)\right\rangle_{L^2(\mathbb{R})}\overline{\left\langle |\cdot|^{\frac{1}{\theta}-1}(\mathfrak{F}_\theta g)(\cdot),(\mathfrak{F}_\theta\psi_{a,b,\theta})(\cdot)\right\rangle_{L^2(\mathbb{R})}}db\notag\\
&=\displaystyle\left(\frac{1}{2\pi\theta}\right)^2\int_{\mathbb{R}}|b|^{\frac{1}{\theta}-1}\left(\int_{\mathbb{R}}|\xi|^{\frac{1}{\theta}-1}(\mathfrak{F}_\theta f)(\xi)\overline{(\mathfrak{F}_\theta\phi_{a,b,\theta})(\xi)}d\xi\right)\overline{\left(\int_{\mathbb{R}}|\omega|^{\frac{1}{\theta}-1}(\mathfrak{F}_\theta f)(\omega)\overline{(\mathfrak{F}_\theta\psi_{a,b,\theta})(\omega)}d\omega\right)}db\notag\\
&=\displaystyle\left(\frac{1}{2\pi\theta}\right)^2|a|^{\frac{1}{\theta}}\int_{\mathbb{R}}|b|^{\frac{1}{\theta}-1}\left(\int_{\mathbb{R}}|\xi|^{\frac{1}{\theta}-1}(\mathfrak{F}_\theta f)(\xi)\overline{e^{-i(\sgn\xi)|\xi|^\frac{1}{\theta}b}(\mathfrak{F}_\theta\phi)(a\xi)}d\xi\right)\notag\\
&\hspace{7.5cm}\times\overline{\left(\int_{\mathbb{R}}|\omega|^{\frac{1}{\theta}-1}(\mathfrak{F}_\theta g)(\omega)\overline{e^{-i(\sgn(\omega))|\omega|^\frac{1}{\theta}b}(\mathfrak{F}_\theta\psi)(a\omega)}d\omega\right)}db\notag\\
&=\displaystyle\left(\frac{1}{2\pi\theta}\right)^2|a|^{\frac{1}{\theta}}\int_{\mathbb{R}}|b|^{\frac{1}{\theta}-1}\overline{\left(\int_{\mathbb{R}}e^{-i(\sgn\xi)|\xi|^\frac{1}{\theta}b}|\xi|^{\frac{1}{\theta}-1}\overline{(\mathfrak{F}_\theta f)(\xi)}(\mathfrak{F}_\theta\phi)(a\xi)d\xi\right)}\notag\\
&\hspace{7.5cm}\times\left(\int_{\mathbb{R}}e^{-i(\sgn(\omega))|\omega|^\frac{1}{\theta}b}|\omega|^{\frac{1}{\theta}-1}\overline{(\mathfrak{F}_\theta g)(\omega)}(\mathfrak{F}_\theta\psi)(a\omega)d\omega\right)db\notag\\
&=\displaystyle\left(\frac{1}{2\pi\theta}\right)^2|a|^{\frac{1}{\theta}}\int_{\mathbb{R}}|b|^{\frac{1}{\theta}-1}\overline{\left(\int_{\mathbb{R}}e^{-i(\sgn\xi)|\xi|^\frac{1}{\theta}b}\overline{P_\theta(\xi)}d\xi\right)}\left(\int_{\mathbb{R}}e^{-i(\sgn(\omega))|\omega|^\frac{1}{\theta}b}\overline{Q_\theta(\omega)}d\omega\right)db\notag\\
&=\displaystyle\left(\frac{1}{2\pi\theta}\right)^2|a|^{\frac{1}{\theta}}\int_{\mathbb{R}}|b|^{\frac{1}{\theta}-1}\overline{\left(\mathfrak{F}_\theta\overline{P_\theta}\right)(b)}\left(\mathfrak{F}_\theta\overline{Q_\theta}\right)(b)db.
\end{align}
Using \cite[Theorem 1]{srivastava2019certain} in equation (\ref{P2eqn22}),  we get
\begin{eqnarray*}
\int_{\mathbb{R}}|b|^{\frac{1}{\theta}-1}(W_\phi^\theta f)(b,a)\overline{(W_\psi^\theta g)(b,a)}db&=&\left(\frac{1}{2\pi\theta}\right)^2|a|^\frac{1}{\theta}\left\langle|\cdot|^{\frac{1}{\theta}-1}\left(\mathfrak{F}_\theta\overline{Q_\theta}\right)(\cdot),\left(\mathfrak{F}_\theta\overline{P_\theta}\right)(\cdot)\right\rangle_{L^2(\mathbb{R})}\\
&=&\left(\frac{1}{2\pi\theta}\right)^2|a|^\frac{1}{\theta}\left\langle\overline{Q_\theta},\overline{P_\theta}\right\rangle_{L^2(\mathbb{R})}\\
&=&\left(\frac{1}{2\pi\theta}\right)^2|a|^\frac{1}{\theta}\left\langle P_\theta,Q_\theta\right\rangle_{L^2(\mathbb{R})}.
\end{eqnarray*}
This completes the proof.
\end{proof}
\section{CFrWT on Hardy space $\&$ Morrey  spaces}
In this section we consider the normalized form of the operator $W_\psi^\theta$ i.e., $L_\psi^\theta :=\frac{1}{\sqrt{C_{\psi,\theta}}}W_\psi^\theta$ and study some of its properties on Hardy space and Morrey space. The purpose of this section is to establish the boundedness of the CFrWT on these spaces and to study the dependence of the CFrWT on its fractional wavelet and its argument function via $H^1(\mathbb{R})$ and $L_M^{1,\nu}(\mathbb{R})-$distance estimate of two CFrWTs with different fractional wavelets of different argument functions.

\subsection{Hardy space}
\begin{defn}(\cite{chuong2013boundedness})
 The Hardy space $H^1(\mathbb{R}),$ defined by
$$H^1(\mathbb{R})=\left\{f\in L^1(\mathbb{R}):\int_{\mathbb{R}}\sup_{t>0}|(f\star\eta_{t})(x)|dx<\infty\right\},$$
is a Banach space normed by
\begin{equation}
\|f\|_{H^1(\mathbb{R})}=\int_{\mathbb{R}}\sup_{t>0}|(f\star\eta_{t})(x)|dx,
\end{equation}\label{P2eqn28}
where $\eta$ is a function in Schwartz space such that $\int_{\mathbb{R}}\eta(x)dx\neq 0$ and $\eta_{t}(x)=\frac{1}{t}\eta(\frac{x}{t}), t>0,x\in\mathbb{R}.$
\end{defn}

We now study some properties of the CFrWT on the Hardy space $H^1(\mathbb{R}).$ To establish the boundedness of the CFrWT on the Hardy space we need to  prove the following  lemma.
\begin{lemma}
Let $a\in\mathbb{R}-\{0\},$ $f\in H^1(\mathbb{R})$ and $\psi$ be a  fractional wavelet, then $(L_\psi^\theta f)(\cdot, a)$ $\in L^1(\mathbb{R}),$ where $L_\psi^\theta f$ denotes the CFrWT of $f.$
\end{lemma}
\begin{proof}
For fix $a\in\mathbb{R}-\{0\},$ $(L_\psi^\theta f)(b,a)$ is a function of $b$ and is such that
\begin{eqnarray*}
|(L_\psi^\theta f)(b,a)|&\leq&\frac{1}{\sqrt{C_{\psi,\theta}}}\int_{\mathbb{R}}|f(u)||\psi_{a,b,\theta}(u)|dt\\
&=&\frac{1}{\sqrt{C_{\psi,\theta}}}\int_{\mathbb{R}}|f(u)|\frac{1}{|a|^{\frac{1}{2\theta}}}\left|\psi\left(\frac{u-b}{(\sgn a)|a|^\frac{1}{\theta}}\right)\right|du\\
&=&\frac{|a|^{\frac{1}{2\theta}}}{\sqrt{C_{\psi,\theta}}}\int_{\mathbb{R}}\left|f\left((\sgn a)|a|^{\frac{1}{\theta}}x+b\right)\right||\psi(x)|dx.
\end{eqnarray*}
Therefore, 
\begin{eqnarray*}
\int_{\mathbb{R}}|(L_\psi^\theta f)(b,a)|db&\leq&\frac{|a|^{\frac{1}{2\theta}}}{\sqrt{C_{\psi,\theta}}}\int_{\mathbb{R}}|\psi(x)|\left(\int_{\mathbb{R}}\left|f\left((\sgn a)|a|^{\frac{1}{\theta}}x+b\right)\right|db\right)dx\\
&=&\frac{|a|^{\frac{1}{2\theta}}}{\sqrt{C_{\psi,\theta}}}\|\psi\|_{L^1(\mathbb{R})}\|f\|_{L^1(\mathbb{R})}.
\end{eqnarray*}
Hence, it follows that $(L_\psi^\theta f)(\cdot,a)\in L^1(\mathbb{R}).$
\end{proof}
\begin{theorem}\label{P2HS1st}
Let $a\in\mathbb{R}-\{0\},$ then the operator $L_\psi^\theta:H^1(\mathbb{R})\rightarrow H^1(\mathbb{R})$ defined by $f\rightarrow (L_\psi^\theta f)(\cdot,a)$ is bounded. Furthermore, $$\|(L_\psi^\theta f)(\cdot,a)\|_{H^1(\mathbb{R})}\leq\frac{|a|^{\frac{1}{2\theta}}}{\sqrt{C_{\psi,\theta}}}\|\psi\|_{L^1(\mathbb{R})}\|f\|_{H^1(\mathbb{R})}.$$
\end{theorem}
\begin{proof}
From definition of $L_\psi^\theta $, we get
\begin{eqnarray*}
(L_\psi^\theta f)(b,a)&=&\frac{|a|^{\frac{1}{2\theta}}}{\sqrt{C_{\psi,\theta}}}\int_{\mathbb{R}}f\left((\sgn a)|a|^{\frac{1}{\theta}}x+b\right)\overline{\psi(x)}dx.\\
\end{eqnarray*}
Now,
\begin{eqnarray*}
((L_\psi^\theta f)(\cdot,a)\star\eta_t(\cdot))(b)&=&\int_{\mathbb{R}}(L_\psi^\theta f)(b-y,a)\eta_t(y)dy\\
&=&\int_{\mathbb{R}}\frac{|a|^{\frac{1}{2\theta}}}{\sqrt{C_{\psi,\theta}}}\left(\int_{\mathbb{R}}f\left((\sgn a)|a|^{\frac{1}{\theta}}x+b-y\right)\overline{\psi(x)}dx\right)\eta_t(y)dy\\
&=&\frac{|a|^{\frac{1}{2\theta}}}{\sqrt{C_{\psi,\theta}}}\int_{\mathbb{R}}\overline{\psi(x)}\left(\int_{\mathbb{R}}f\left((\sgn a)|a|^{\frac{1}{\theta}}x+b-y\right)\eta_t(y)dy\right)dx\\
&=&\frac{|a|^{\frac{1}{2\theta}}}{\sqrt{C_{\psi,\theta}}}\int_{\mathbb{R}}(f\star\eta_t)\left(\sgn a|a|^{\frac{1}{\theta}}x+b\right)\overline{\psi(x)}dx.\\
\end{eqnarray*}
Therefore, 
\begin{eqnarray*}
\|(L_\psi^\theta f)(\cdot,a)\|_{H^1(\mathbb{R})}&=&\int_{\mathbb{R}}\sup_{t>0}|\left((L_\psi^\theta f)(\cdot,a)\star\eta_t(\cdot)\right)(b)|db\\
&\leq &\frac{|a|^{\frac{1}{2\theta}}}{\sqrt{C_{\psi,\theta}}}\int_{\mathbb{R}}|\overline{\psi(x)}|\left(\int_{\mathbb{R}}\sup_{t>0}|(f\star\eta_t)((\sgn a)|a|^{\frac{1}{\theta}}x+b)|db\right)dx.\\
&=&\frac{|a|^{\frac{1}{2\theta}}}{\sqrt{C_{\psi,\theta}}}\|\psi\|_{L^1(\mathbb{R})}\|f\|_{H^1(\mathbb{R})}.
\end{eqnarray*}
This completes the proof.
\end{proof}
\begin{corollary}
If $a\in\mathbb{R}-\{0\},$ $f\in H^1(\mathbb{R})$ and $\psi$ is a fractional wavelet,  then $\|(L_\psi^\theta f)(\cdot,a)\|_{H^1(\mathbb{R})}=O\left(|a|^{\frac{1}{2\theta}}\right).$
\end{corollary}
\begin{proof}
By using the theorem \ref{P2HS1st}, we get the result.
\end{proof}
We will now determine the $H^1(\mathbb{R})-$distance of two CFrWTs with different fractional wavelets and different argument functions to study the dependence of the transform on its fractional wavelet and its argument.
\begin{theorem}\label{P2HS2nd}
Let $f,g\in H^1(\mathbb{R})$ and $\phi,\psi$ be two fractional wavelets then, for $a\in\mathbb{R}-\{0\},$
$$\|(L_\phi^\theta f)(\cdot,a)-(L_\psi^\theta g)(\cdot,a)\|_{H^1(\mathbb{R})}\leq |a|^{\frac{1}{2\theta}}\left(\|f\|_{H^1(\mathbb{R})}\left\|\frac{\phi}{\sqrt{C_{\phi,\theta}}}-\frac{\psi}{
\sqrt{C_{\psi,\theta}}}\right\|_{L^1(\mathbb{R})}+\|f-g\|_{H^1(\mathbb{R})}\left\|\frac{\psi}{\sqrt{C_{\psi,\theta}}}\right\|_{L^1(\mathbb{R})} \right).$$
\end{theorem}
\begin{proof}
We have
\begin{equation}\label{P2eqn29}
\|(L_\phi^\theta f)(\cdot,a)-(L_\psi^\theta g)(\cdot,a)\|_{H^1(\mathbb{R})}\leq \|(L_\phi^\theta f)(\cdot,a)-(L_\psi^\theta f)(\cdot,a)\|_{H^1(\mathbb{R})}+\|(L_\psi^\theta f)(\cdot,a)-(L_\psi^\theta g)(\cdot,a)\|_{H^1(\mathbb{R})}.
\end{equation}
Now,
\begin{eqnarray*}
 (L_\phi^\theta f)(b,a)-(L_\psi^\theta f)(b,a)&=&\frac{1}{|a|^{\frac{1}{2\theta}}}\int_{\mathbb{R}}f(t)\left\{\overline{\frac{1}{\sqrt{C_{\phi,\theta}}}\phi\left(\frac{t-b}{(\sgn a)|a|^{\frac{1}{\theta}}}\right)}-\overline{\frac{1}{\sqrt{C_{\psi,\theta}}}\psi\left(\frac{t-b}{(\sgn a)|a|^{\frac{1}{\theta}}}\right)}\right\}dt\\
&=& |a|^{\frac{1}{2\theta}}\int_{\mathbb{R}}f\left((\sgn a)|a|^{\frac{1}{\theta}}x+b\right)\left(\overline{\frac{\phi(x)}{\sqrt{C_{\phi,\theta}}}-\frac{\psi(x)}{
\sqrt{C_{\psi,\theta}}}}\right)dx.\\
\end{eqnarray*}
Observe that
\begin{eqnarray*}
\left\{\left((L_\phi^\theta f)(\cdot,a)-(L_\phi^\theta f)(\cdot,a)\right)\star\eta_t(\cdot)\right\}(b)&=&\int_{\mathbb{R}}\left\{\left((L_\phi^\theta f)(b-y,a)-(L_\phi^\theta f)(b-y,a)\right)\right\}\eta_t(y)dy\\
&=&\int_{\mathbb{R}}\left(|a|^{\frac{1}{2\theta}}\int_{\mathbb{R}}f\left((\sgn a)|a|^{\frac{1}{\theta}}x+b-y\right)\left(\overline{\frac{\phi(x)}{\sqrt{C_{\phi,\theta}}}-\frac{\psi(x)}{
\sqrt{C_{\psi,\theta}}}}\right)dx\right)\eta_t(y)dy\\
&=&|a|^{\frac{1}{2\theta}}\int_{\mathbb{R}}\left(\overline{\frac{\phi(x)}{\sqrt{C_{\phi,\theta}}}-\frac{\psi(x)}{
\sqrt{C_{\psi,\theta}}}}\right)\left(\int_{\mathbb{R}}f\left((\sgn a)|a|^{\frac{1}{\theta}}x+b-y\right)\eta_t(y)dy\right)dx\\
&=&|a|^{\frac{1}{2\theta}}\int_{\mathbb{R}}\left(f\star\eta_t\right)\left((\sgn a)|a|^{\frac{1}{\theta}}x+b\right)\left(\overline{\frac{\phi(x)}{\sqrt{C_{\phi,\theta}}}-\frac{\psi(x)}{
\sqrt{C_{\psi,\theta}}}}\right)dx.
\end{eqnarray*}
Therefore,
\begin{eqnarray}\label{P2eqn30}
\|(L_\phi^\theta f)(\cdot,a)-(L_\psi^\theta f)(\cdot,a)\|_{H^1(\mathbb{R})}\notag &=&\int_{\mathbb{R}}\sup_{t>0}\left|\left\{\left((L_\phi^\theta f)(\cdot,a)-(L_\phi^\theta f)(\cdot,a)\right)\star\eta_t(\cdot)\right\}(b)\right|db\\
\notag 
&\leq &|a|^{\frac{1}{2\theta}}\int_{\mathbb{R}}\left|\frac{\phi(x)}{\sqrt{C_{\phi,\theta}}}-\frac{\psi(x)}{
\sqrt{C_{\psi,\theta}}}\right|\left(\int_{\mathbb{R}}\sup_{t>0}|\left(f\star\eta_t\right)\left((\sgn a)|a|^{\frac{1}{\theta}}x+b\right)|db\right)dx\\
&=& |a|^{\frac{1}{2\theta}}\|f\|_{H^1(\mathbb{R})}\left\|\frac{\phi}{\sqrt{C_{\phi,\theta}}}-\frac{\psi}{
\sqrt{C_{\psi,\theta}}}\right\|_{L^1(\mathbb{R})}.
\end{eqnarray}
Similarly, it can be shown that
\begin{equation}\label{P2eqn31}
\|(L_\psi^\theta f)(\cdot,a)-(L_\psi^\theta g)(\cdot,a)\|_{H^1(\mathbb{R})}\leq |a|^{\frac{1}{2\theta}}\|f-g\|_{H^1(\mathbb{R})}\left\|\frac{\psi}{\sqrt{C_{\psi,\theta}}}\right\|_{L^1(\mathbb{R})}
\end{equation}
From equations (\ref{P2eqn29}),(\ref{P2eqn30}) and (\ref{P2eqn31}) the result follows.
\end{proof}

\begin{remark}
For $\theta=1$ the theorem \ref{P2HS1st} and theorem \ref{P2HS2nd}  coincide with those studied in \cite{chuong2013boundedness}.
\end{remark}
\subsection{Morrey space}
\begin{defn}(\cite{almeida2017approximation})
The Morrey space $L^{p,\nu}_{M}(\mathbb{R}),$ with $1\leq p<\infty$ and $0\leq \nu\leq 1,$ defined by
\begin{equation*}
L^{p,\nu}_{M}(\mathbb{R})=\left\{f\in L_{loc}^p(\mathbb{R}): \sup_{\substack{x\in\mathbb{R}\\r>0}}\left(\frac{1}{r^{\nu}}\int_{B(x,r)}|f(t)|^pdt\right)<\infty\right\},
\end{equation*}
is a Banach space normed by
\begin{equation*}
\|f\|_{L^{p,\nu}_{M}(\mathbb{R})}=\sup_{\substack{x\in\mathbb{R}\\r>0}}\left(\frac{1}{r^{\nu}}\int_{B(x,r)}|f(t)|^pdt\right)^{\frac{1}{p}}.
\end{equation*}
\end{defn}

We now study some properties of the CFrWT on the Morrey space $L^{1,\nu}_{M}(\mathbb{R}).$ Before we establish the boundedness of the CFrWT on the Morrey space we prove the following lemma.
\begin{lemma}\label{P2}
Let $a\in\mathbb{R}-\{0\},$ $f\in L^{1,\nu}_{M}(\mathbb{R})$ and $\psi$ be a compactly supported fractional wavelet, then $(L_\psi^\theta f)(\cdot,a)$ is in $L^1_{loc}(\mathbb{R}).$
\end{lemma}
\begin{proof}
We have
$$(L_\psi^\theta f)(b,a)=\frac{1}{\sqrt{C_{\psi,\theta}}}\int_{\mathbb{R}} f(t)\overline{\psi_{a,b,\theta}(t)}dt.$$
This implies,
\begin{eqnarray*}
|(L_\psi^\theta f)(b,a)|\leq \frac{1}{\sqrt{C_{\psi,\theta}}}\int_{\mathbb{R}}|f(t)||\psi_{a,b,\theta}(t)|dt.
\end{eqnarray*}
Using equation (\ref{P2eqn14.}), we get
\begin{align}\label{P2eqn33}
|(L_\psi^\theta f)(b,a)|&\leq  \frac{1}{|a|^{\frac{1}{2\theta}}\sqrt{C_{\psi,\theta}}}\int_{\mathbb{R}}|f(t)||\psi\left(\frac{t-b}{(\sgn a)|a|^{\frac{1}{\theta}}}\right)|dt\notag\\
&=\frac{|a|^{\frac{1}{2\theta}}}{\sqrt{C_{\psi,\theta}}}\int_{\mathbb{R}}\left|f\left((\sgn a)|a|^\frac{1}{\theta} x+b\right)\right||\psi(x)|dx.
\end{align}
Let $K$ be a compact set in $\mathbb{R}.$ We have
\begin{align*}
\int_{K}|(L_\psi^\theta f)(b,a)|db&\leq \frac{|a|^\frac{1}{2\theta}}{\sqrt{C_{\psi,\theta}}}\int_{K}\int_{\mathbb{R}}\left|f\left((\sgn a)|a|^\frac{1}{\theta} x+b\right)\right||\psi(x)|dx db\\
&=\frac{|a|^\frac{1}{2\theta}}{\sqrt{C_{\psi,\theta}}}\int_{\mathbb{R}}|\psi(x)|\left(\int_{K}\left|f\left(\sgn a|a|^\frac{1}{\theta} x+b\right)\right|db\right) dx\\
&=\frac{|a|^\frac{1}{2\theta}}{\sqrt{C_{\psi,\theta}}}\int_{\mathbb{R}}|\psi(x)|\left(\int_{A}|f(y)|dy\right)dx,
\end{align*}
where $A=(\sgn a)|a|^{\frac{1}{\theta}}x+K\subset(\sgn a)|a|^{\frac{1}{\theta}}\mbox{supp}~\psi+K.$ Since $f\in L^1_{loc}(\mathbb{R})$ and $A$ is bounded, we have
$$\int_{K}|(L_\psi^\theta f)(b,a)|db \leq \frac{|a|^\frac{1}{2\theta}}{\sqrt{C_{\psi,\theta}}}C\|\psi\|_{L^1(\mathbb{R})},$$ for some $C\geq 0.$ Thus, it follow that $(L_\psi^\theta f)(\cdot,a)$ is in $L^1_{loc}(\mathbb{R}).$
\end{proof}
\begin{theorem}\label{P2theo4.4}
Let $a\in\mathbb{R}-\{0\}$ and $\psi$ be a compactly supported fractional wavelet, then the operator $L_\psi^\theta:L^{1,\nu}_{M}(\mathbb{R})\rightarrow L^{1,\nu}_{M}(\mathbb{R})$ defined by $f\rightarrow (L_\psi^\theta f)(\cdot,a)$ is bounded. Furthermore,
$$\|(L_\psi^\theta f)(\cdot,a)\|_{L^{1,\nu}_{M}(\mathbb{R})}\leq\frac{|a|^\frac{1}{2\theta}}{\sqrt{C_{\psi,\theta}}}\|\psi\|_{L^1(\mathbb{R})}\|f\|_{L^{1,\nu}_{M}(\mathbb{R})}.$$
\end{theorem}
\begin{proof}
We have
\begin{equation}\label{P2eqn34}
\|(L_\psi^\theta f)(\cdot,a)\|_{L^{1,\nu}_{M}(\mathbb{R})}=\sup_{\substack{x\in\mathbb{R}\\r>0}}\left(\frac{1}{r^{\nu}}\int_{B(x,r)}|(L_\psi^\theta f)(b,a)|db\right).
\end{equation}
Now using equation (\ref{P2eqn33}), we get
\begin{eqnarray}\label{P2eqn35}
\frac{1}{r^{\nu}}\int_{B(x,r)}|(L_\psi^\theta f)(b,a)|db&\leq &\frac{|a|^\frac{1}{2\theta}}{r^{\nu}\sqrt{C_{\psi,\theta}}}\int_{B(x,r)}\left(\int_{\mathbb{R}}\left|f\left((\sgn a)|a|^\frac{1}{\theta} u+b\right)\right||\psi(u)|du\right)db\notag\\
&= &\frac{|a|^\frac{1}{2\theta}}{r^{\nu}\sqrt{C_{\psi,\theta}}}\int_{\mathbb{R}}|\psi(u)|\left(\int_{B(x,r)}\left|f\left((\sgn a)|a|^\frac{1}{\theta} u+b\right)\right|db\right)du\notag\\
&=&\frac{|a|^\frac{1}{2\theta}}{\sqrt{C_{\psi,\theta}}}\int_{\mathbb{R}}|\psi(u)|\left(\frac{1}{r^{\nu}}\int_{B\left((\sgn a)|a|^\frac{1}{\theta}u+x,r\right)}|f(\boldsymbol z)|d\boldsymbol z\right)du.
\end{eqnarray}
Also 
\begin{eqnarray}\label{P2eqn36}
\frac{1}{r^{\nu}}\int_{B\left((\sgn a)|a|^{\frac{1}{\theta}}u+x,r\right)}|f(\boldsymbol z)|d\boldsymbol z &\leq& \|f\|_{L^{1,\nu}_{M}(\mathbb{R}).}
\end{eqnarray}
From equations (\ref{P2eqn35}) and (\ref{P2eqn36}), we have
$$\frac{1}{r^{\nu}}\int_{B(x,r)}|(L_\psi^\theta f)(b,a)|db\leq\frac{|a|^\frac{1}{2\theta}}{\sqrt{C_{\psi,\theta}}}\|f\|_{L^{1,\nu}_{M}(\mathbb{R})}\|\psi\|_{L^1(\mathbb{R})},$$
which gives
\begin{equation}\label{P2eqn37}
\sup_{\substack{x\in\mathbb{R}\\r>0}}\left(\frac{1}{r^{\nu}}\int_{B(x,r)}|(L_\psi^\theta f)(b,a)|db\right)\leq\frac{|a|^\frac{1}{2\theta}}{\sqrt{C_{\psi,\theta}}}\|f\|_{L^{1,\nu}_{M}(\mathbb{R})}\|\psi\|_{L^1(\mathbb{R})}.
\end{equation}
From (\ref{P2eqn34}) and (\ref{P2eqn37}), we have
$$\|(L_\psi^\theta f)(\cdot,a)\|_{L^{1,\nu}_{M}(\mathbb{R})}\leq\frac{|a|^\frac{1}{2\theta}}{\sqrt{C_{\psi,\theta}}}\|\psi\|_{L^1(\mathbb{R})}\|f\|_{L^{1,\nu}_{M}(\mathbb{R})}.$$
This completes the proof.
\end{proof}
\begin{corollary}\label{P2corollary}
Let $a\in\mathbb{R}-\{0\},$ $f\in L^{1,\nu}_{M}(\mathbb{R}),$ and $\psi$ be a compactly supported fractional wavelet, then $$\|(L_\psi^\theta f)(\cdot,a)\|_{L^{1,\nu}_{M}(\mathbb{R})}=O\left({|a|^\frac{1}{2\theta}}\right).$$ 
\end{corollary}
\begin{proof}
Follows from theorem \ref{P2theo4.4}.
\end{proof}
We will now determine the $L^{1,\nu}_{M}(\mathbb{R})-$distance of two CFrWTs with different fractional wavelets and different argument functions to study the dependence of the transform on its fractional wavelet and its argument.
\begin{theorem}\label{P2theo}
Let $f,g\in L^{1,\nu}_{M}(\mathbb{R})$ and $\phi,\psi$ be two compactly supported fractional wavelets. Then for $a\in\mathbb{R}-\{0\},$
$$\|(L_\phi^\theta f)(\cdot,a)-(L_\psi^\theta g)(\cdot,a)\|_{L^{1,\nu}_{M}(\mathbb{R})}\leq |a|^\frac{1}{2\theta}\left(\|f\|_{L^{1,\nu}_{M}(\mathbb{R})}\left\|\frac{\phi}{\sqrt{C_{\phi,\theta}}}-\frac{\psi}{\sqrt{C_{\psi,\theta}}}\right\|_{L^1(\mathbb{R})}+\|f-g\|_{L^{1,\nu}_{M}(\mathbb{R})}\left\|\frac{\psi}{\sqrt{C_{\psi,\theta}}}\right\|_{L^1(\mathbb{R})}\right).$$
\end{theorem}
\begin{proof}
We have
\begin{align}\label{P2eqn38.}
\notag\|(L_\phi^\theta f)(\cdot,a)-(L_\psi^\theta g)(\cdot,a)&\|_{L^{1,\nu}_{M}(\mathbb{R})}\\
&\leq\|(L_\phi^\theta f)(\cdot,a)-(L_\psi^\theta f)(\cdot,a)\|_{L^{1,\nu}_{M}(\mathbb{R})}+\|(L_\psi^\theta f)(\cdot,a)-(L_\psi^\theta g)(\cdot,a)\|_{L^{1,\nu}_{M}(\mathbb{R})}.
\end{align}
Here, 
\begin{equation}{\label{P2eqn38}}
\|(L_\phi^\theta f)(\cdot,a)-(L_\psi^\theta f)(\cdot,a)\|_{L^{1,\nu}_{M}(\mathbb{R})}=\sup_{\substack{x\in\mathbb{R}\\r>0}}\left(\frac{1}{r^{\nu}}\int_{B(x,r)}\left|\int_{\mathbb{R}}\left(\frac{1}{\sqrt{C_{\phi,\theta}}}f(y)\overline{\phi_{a,t,\theta}(y)}-\frac{1}{\sqrt{C_{\psi,\theta}}}f(y)\overline{\psi_{a,t,\theta}(y)}\right)dy\right| dt\right).
\end{equation}
Now,
\begin{align*}
\frac{1}{r^{\nu}}\int_{B(x,r)}&\left|\int_{\mathbb{R}}\left(\frac{1}{\sqrt{C_{\phi,\theta}}}f(y)\overline{\phi_{a,t,\theta}(y)}-\frac{1}{\sqrt{C_{\psi,\theta}}}f(y)\overline{\psi_{a,t,\theta}(y)}\right)dy\right| dt \\
&\leq  \frac{1}{r^{\nu}}\int_{B(x,r)}\int_{\mathbb{R}}|f(y)|\left|\left(\frac{1}{\sqrt{C_{\phi,\theta}}}\phi_{a,t,\theta}(y)-\frac{1}{\sqrt{C_{\psi,\theta}}}\psi_{a,t,\theta}(y)\right)\right|dy dt\\
&= \frac{1}{r^{\nu}}\int_{B(x,r)}\bigg(\int_{\mathbb{R}}\frac{1}{|a|^\frac{1}{2\theta}}|f(y)|\left|\frac{1}{\sqrt{C_{\phi,\theta}}}\phi\bigg(\frac{y-t}{(\sgn a)|a|^\frac{1}{\theta}}\bigg)-\frac{1}{\sqrt{C_{\psi,\theta}}}\psi\left(\frac{y-t}{(\sgn a)|a|^\frac{1}{\theta}}\right)\right|dy\bigg) dt\\
&= \frac{|a|^\frac{1}{2\theta}}{r^{\nu}}\int_{B(x,r)}\left(\int_{\mathbb{R}}\left|f\left((\sgn a)|a|^\frac{1}{\theta}z+t\right)\right|\left|\frac{1}{\sqrt{C_{\phi,\theta}}}\phi(z)-\frac{1}{\sqrt{C_{\psi,\theta}}}\psi(z)\right|dz\right) dt\\
&=\frac{|a|^\frac{1}{2\theta}}{r^{\nu}}\int_{\mathbb{R}}\left|\frac{1}{\sqrt{C_{\phi,\theta}}}\phi(z)-\frac{1}{\sqrt{C_{\psi,\theta}}}\psi(z)\right|\left(\int_{B(x,r)}\left|f\left((\sgn a)|a|^\frac{1}{\theta}z+t\right)\right|dt\right)dz\\
&=|a|^\frac{1}{2\theta}\int_{\mathbb{R}}\left|\frac{1}{\sqrt{C_{\phi,\theta}}}\phi(z)-\frac{1}{\sqrt{C_{\psi,\theta}}}\psi(z)\right|\left(\frac{1}{r^{\nu}}\int_{B\left((\sgn a)|a|^\frac{1}{\theta}z+x,r\right)}|f(u)|du\right)dz.
\end{align*}
Using equation (\ref{P2eqn36}), we get
$$\frac{1}{r^{\nu}}\int_{B(x,r)}\left|\int_{\mathbb{R}}\left(\frac{1}{\sqrt{C_{\phi,\theta}}}f(y)\overline{\phi_{a,t,\theta}(y)}-\frac{1}{\sqrt{C_{\phi,\theta}}}f(y)\overline{\psi_{a,t,\theta}(y)}\right)dy\right| dt\leq |a|^\frac{1}{2\theta}\|f\|_{L^{1,\nu}_{M}(\mathbb{R})}\int_{\mathbb{R}}\left|\frac{1}{\sqrt{C_{\phi,\theta}}}\phi(z)-\frac{1}{\sqrt{C_{\psi,\theta}}}\psi(z)\right|dz.$$
Therefore,
\begin{equation}\label{P2eqn39}
\sup_{\substack{x\in\mathbb{R}\\r>0}}\left(\frac{1}{r^{\nu}}\int_{B(x,r)}\left|\int_{\mathbb{R}}\big(f(y)\overline{\phi_{a,t,\theta}(y)}-f(y)\overline{\psi_{a,t,\theta}(y)}\big)dy\right| dt\right)\leq |a|^\frac{1}{2\theta}\|f\|_{L^{1,\nu}_{M}(\mathbb{R})}\left\|\frac{\phi}{\sqrt{C_{\phi,\theta}}}-\frac{\psi}{\sqrt{C_{\phi,\theta}}}\right\|_{L^1(\mathbb{R})}.
\end{equation}
From equation (\ref{P2eqn38}) and equation (\ref{P2eqn39}), it follows that 
\begin{equation}\label{P2eqn40}
\|(L_\phi^\theta f)(\cdot,a)-(L_\psi^\theta f)(\cdot,a)\|_{L^{1,\nu}_{M}(\mathbb{R})}\leq |a|^\frac{1}{2\theta}\|f\|_{L^{1,\nu}_{M}(\mathbb{R})}\left\|\frac{\phi}{\sqrt{C_{\phi,\theta}}}-\frac{\psi}{\sqrt{C_{\psi,\theta}}}\right\|_{L^1(\mathbb{R})}.
\end{equation}
Similarly, it can be shown that 
\begin{equation}\label{P2eqn41}
\|(L_\psi^\theta f)(\cdot,a)-(L_\psi^\theta g)(\cdot,a)\|_{L^{1,\nu}_{M}(\mathbb{R})}\leq |a|^\frac{1}{2\theta}\|f-g\|_{L^{1,\nu}_{M}(\mathbb{R})}\left\|\frac{\psi}{\sqrt{C_{\psi,\theta}}}\right\|_{L^1(\mathbb{R})}
\end{equation}
From equations (\ref{P2eqn38.}), (\ref{P2eqn40}) and (\ref{P2eqn41}) the theorem follows immediately.
\end{proof}
\section{Conclusions}
In this paper, we have studied CFrWT which as a generalization of the classical wavelet transform, reduces to classical wavelet transform for $\theta=1.$ In section 2, we have introduced some basic definitions and results. In section 3, we have generalized the existing result like orthogonality relation, reconstruction formula and the range theorem, in \cite{srivastava2019certain,upadhyayCFrWT}, in the context of two fractional wavelets. Also we have derived the formulas for the CFrWT when the argument function or fractional wavelet is a convolution or correlation of two functions. Lastly, in section 4, the boundedness of CFrWT on Hardy space $H^1(\mathbb{R})$ and  Morrey space $L_M^{1,\nu}(\mathbb{R})$ along with its approximation property are established.  
\section{Acknowledgement} The work is partially supported by UGC File No. 16-9(June 2017)/2018(NET/CSIR), New Delhi, India.

\bibliography{P2MasterS2Revised}

\begin{thebibliography}{10}

\bibitem{kilbas2010fractional}
H.~Martinez A.~A.~Kilbas, Y. F.~Luchko and J.~J. Trujillo.
\newblock Fractional {F}ourier transform in the framework of fractional
  calculus operators.
\newblock {\em Integral Transforms and Special Functions}, 21(10):779--795,
  2010.

\bibitem{prasad2014generalized}
A.~Mahato A.~Prasad, S.~Manna and V.~K. Singh.
\newblock The generalized continuous wavelet transform associated with the
  fractional {F}ourier transform.
\newblock {\em Journal of computational and applied mathematics}, 259:660--671,
  2014.

\bibitem{almeida2017approximation}
A.~Almeida and S.~Samko.
\newblock Approximation in {M}orrey spaces.
\newblock {\em Journal of Functional Analysis}, 272(6):2392--2411, 2017.

\bibitem{almeida1994fractional}
L.B. Almeida.
\newblock The fractional {F}ourier transform and time-frequency
  representations.
\newblock {\em IEEE Transactions on signal processing}, 42(11):3084--3091,
  1994.

\bibitem{chuong2013boundedness}
N.~M. Chuong and D.~V. Duong.
\newblock Boundedness of the wavelet integral operator on weighted function
  spaces.
\newblock {\em Russian Journal of Mathematical Physics}, 20(3):268--275, 2013.

\bibitem{mendlovic1997fractional}
D.~Mas J.~Garc{\'\i}a D.~Mendlovic, Z.~Zalevsky and C.~Ferreira.
\newblock Fractional wavelet transform.
\newblock {\em Applied optics}, 36(20):4801--4806, 1997.

\bibitem{daubechies1992ten}
I.~Daubechies.
\newblock Ten lectures on wavelets, volume 61 of cbms-nsf regional conference
  series in applied mathematics (society for industrial and applied mathematics
  (siam).
\newblock 1992.

\bibitem{guo2018linear}
Y.~Guo and B-Z. Li.
\newblock The linear canonical wavelet transform on some function spaces.
\newblock {\em International Journal of Wavelets, Multiresolution and
  Information Processing}, 16(01):1850010, 2018.

\bibitem{dai2017new}
Z.~Zheng H.~Dai and W.~Wang.
\newblock A new fractional wavelet transform.
\newblock {\em Communications in Nonlinear Science and Numerical Simulation},
  44:19--36, 2017.

\bibitem{srivastava2019certain}
K.~Khatterwani H.~M.~Srivastava and S.~K. Upadhyay.
\newblock A certain family of fractional wavelet transformations.
\newblock {\em Mathematical Methods in the Applied Sciences}, 42(9):3103--3122,
  2019.

\bibitem{srivastava2017family}
S.~K.~Upadhyay H.~M.~Srivastava and K.~Khatterwani.
\newblock A family of pseudo-differential operators on the schwartz space
  associated with the fractional {F}ourier transform.
\newblock {\em Russian Journal of Mathematical Physics}, 24(4):534--543, 2017.

\bibitem{shi2012novel}
N.~T.~Zhang J.~Shi and X.~P. Liu.
\newblock A novel fractional wavelet transform and its applications.
\newblock {\em Science China Information Sciences}, 55(6):1270--1279, 2012.

\bibitem{upadhyayCFrWT}
K.~Khatterwani and S.~K. Upadhyay.
\newblock Continuous fractional wavelet transform.
\newblock {\em J. Int. Acad. Phys. Sci}, 21(1):55--61, 2018.

\bibitem{mcbride1987namias}
A.C. McBride and F.H. Kerr.
\newblock On {N}amias's fractional {F}ourier transforms.
\newblock {\em IMA Journal of applied mathematics}, 39(2):159--175, 1987.

\bibitem{meyer1992wavelets}
Y.~Meyer.
\newblock {\em Wavelets and operators}, volume~1.
\newblock Cambridge university press, 1992.

\bibitem{namias1980fractional}
V.~Namias.
\newblock The fractional order {F}ourier transform and its application to
  quantum mechanics.
\newblock {\em IMA Journal of Applied Mathematics}, 25(3):241--265, 1980.

\bibitem{pathak2004wavelet}
R.~S. Pathak.
\newblock The wavelet transform of distributions.
\newblock {\em Tohoku Mathematical Journal, Second Series}, 56(3):411--421,
  2004.

\bibitem{pathak2006wavelet}
R.~S. Pathak and S.~K. Singh.
\newblock The wavelet transform on spaces of type {S}.
\newblock {\em Proceedings of the Royal Society of Edinburgh Section A:
  Mathematics}, 136(4):837--850, 2006.

\bibitem{pathak2007boundedness}
R.~S. Pathak and S.~K. Singh.
\newblock Boundedness of the wavelet transform in certain function spaces.
\newblock {\em J. Inequal. Pure Appl. Math}, 8(1):8, 2007.

\bibitem{prasad2015continuous}
A.~Prasad and P.~Kumar.
\newblock The continuous fractional wavelet transform on generalized weighted
  {S}obolev spaces.
\newblock {\em Asian-European Journal of Mathematics}, 8(03):1550054, 2015.

\bibitem{prasad2016continuous}
A.~Prasad and P.~Kumar.
\newblock The continuous fractional wavelet transform on a generalized
  {S}obolev space.
\newblock {\em International Journal of Wavelets, Multiresolution and
  Information Processing}, 14(06):1650046, 2016.

\bibitem{prasad2016fractional}
A.~Prasad and P.~Kumar.
\newblock Fractional continuous wavelet transform on some function spaces.
\newblock {\em Proceedings of the National Academy of Sciences, India Section
  A: Physical Sciences}, 86(1):57--64, 2016.

\bibitem{prasad2012fractional}
A.~Prasad and A.~Mahato.
\newblock The fractional wavelet transform on spaces of type {S}.
\newblock {\em Integral Transforms and Special Functions}, 23(4):237--249,
  2012.

\bibitem{prasad2013fractional}
A.~Prasad and A.~Mahato.
\newblock The fractional wavelet transform on spaces of type {W}.
\newblock {\em Integral Transforms and Special Functions}, 24(3):239--250,
  2013.

\bibitem{rieder1990wavelet}
A.~Rieder.
\newblock The wavelet transform on {S}obolev spaces and its approximation
  properties.
\newblock {\em Numerische Mathematik}, 58(1):875--894, 1990.

\bibitem{pilipovic2016wavelet}
N.~Teofanov S.~Pilipovi{\'c}, D.~Raki{\'c} and J.~Vindas.
\newblock The wavelet transforms in {G}elfand--{S}hilov spaces.
\newblock {\em Collectanea Mathematica}, 67(3):443--460, 2016.

\bibitem{wei2014generalized}
D.~Wei and Y-M. Li.
\newblock Generalized wavelet transform based on the convolution operator in
  the linear canonical transform domain.
\newblock {\em Optik-International Journal for Light and Electron Optics},
  125(16):4491--4496, 2014.

\bibitem{luchko2008fractional}
H.~Martinez Y.~F.~Luchko and J.~J. Trujillo.
\newblock Fractional fourier transform and some of its applications.
\newblock {\em Fract. Calc. Appl. Anal}, 11(4):457--470, 2008.

\end{thebibliography}
\bibliographystyle{plain}
\end{document}